\documentclass[a4paper,11pt]{amsart}

\usepackage{amssymb,amsbsy,amsmath,amsfonts,amssymb,amscd}
\usepackage{latexsym}
\usepackage{graphics}
\usepackage{color}
\usepackage{comment}
\input xy
\xyoption{all}

\theoremstyle{plain}
\newtheorem{thm}{Theorem}[section]

\newtheorem{lemma}[thm]{Lemma}

\theoremstyle{definition}

\newtheorem{example}[thm]{Example}

\numberwithin{equation}{section}

\newcommand{\CC}{\ensuremath{\mathbb{C}}}

\newcommand{\ZZ}{\ensuremath{\mathbb{Z}}}

\newcommand\om{\omega}
\newcommand\la{\lambda}

\newcommand{\Lam}{\Lambda}

\DeclareMathOperator{\GL}{GL}

\DeclareMathOperator{\im}{im}

\newcommand{\HHH}{\ensuremath{\mathcal{H}}}

\def\eea{\end{eqnarray*}}
\def\bea{\begin{eqnarray*}}

\newcommand\dual{\mathrel{\raise3pt\hbox{$\underline{\mathrm{\thinspace d
\thinspace}}$}}}
\newcommand\qe{\ifhmode\unskip\nobreak\fi\quad $\Box$}       

\def\BOX{\hfill\lower.5\baselineskip\hbox{$\Box$}}

\newtheorem{theo}{Theorem}[section]

\newtheorem{remarkk}[theo]{Remark}

\newtheorem{prop}[theo] {Proposition}

\usepackage{hyperref}
\title [Hyperelliptic Threefolds]{The classification of Hyperelliptic threefolds}

\author{Fabrizio Catanese and Andreas Demleitner}
\address {Lehrstuhl Mathematik VIII\\
Mathematisches Institut der Universit\"at Bayreuth\\
NW II,  Universit\"atsstr. 30\\
95447 Bayreuth}
\email{Fabrizio.Catanese@uni-bayreuth.de \newline \hspace*{2.2cm} Andreas.Demleitner@uni-bayreuth.de}

\thanks{
The present work took place in the framework of the 
 ERC Advanced grant n. 340258, `TADMICAMT' }

\date{\today}

\begin{document}

\maketitle

\begin{abstract}
We complete the classification of  hyperelliptic threefolds, describing in an elementary way  the hyperelliptic threefolds with group $D_4$.
These are algebraic and form an irreducible 2-dimensional family.

\end{abstract}


\section*{Introduction}

A Generalized Hyperelliptic Manifold $X$  is defined to be a quotient $ X = T / G$ of a complex torus $T$ by the free action of a finite group $G$ which contains no translations. 
We say that $X$ is a Generalized Hyperelliptic Variety if moreover the torus $T$ is projective, i.e., it is an Abelian variety $A$.

 The main purpose of the present paper is to complete the classification of the Generalized Hyperelliptic Manifolds of complex dimension three. The cases where the group $G$ is Abelian were classified by   H. Lange in  \cite{Lange}, using work of Fujiki \cite{Fujiki} and  the classification
 of the possible groups $G$ given by Uchida and Yoshihara in \cite{U-Y}:  the latter authors showed that the only possible
 non Abelian group is  the dihedral group $D_4$ of order $8$.
 
 This case was first excluded but it was later found that it does indeed occur (see \cite{C-D} for an account of the story and of the role of the paper \cite{Dekimpe}). Our paper is fully self-contained
 and show that the family described in \cite{C-D} gives all the possible hyperelliptic threefolds with group $D_4$. \\
 
 Our main theorem is the following
 
\begin{thm}\label{Dihedral}
	Let  $T$ be a complex torus of dimension $3$ admitting a  fixed point free action of the dihedral group $$G : = D_4 : = \langle r,s | r^4=1, s^2=1, (rs)^2 = 1\rangle,$$
	such that $G= D_4$ contains no translations.
	
	Then $T$ is algebraic. More precisely, there are two elliptic curves  $E, E'$ such that:
	
	(I) $T$ is a quotient $T  : = T' /  H, \  H \cong \ZZ/2$, where  
	$$T' : = E \times E \times E' = : E_1 \times E_2 \times E_3,\ \     $$
	$$  H : =  \langle \om\rangle , \ \   \om : = \left(h+k, h+k, 0\right) \in T' [2],$$
	
	and $h,k$ are 2-torsion element $h, k \in E[2]$, such that $h,k \neq 0, h + k  \neq 0$;

	(II) there is an element  $h' \in E'$   of order precisely $4$, such that, for $z = (z_1,z_2,z_3) \in T'$:
\begin{align*}
&r(z) = (z_2, - z_1, z_3+ h') = R(z_1,z_2,z_3) + \left(0,0,h'\right), \\
&s(z) = \left(z_1+ h , -z_2 + k , -z_3\right) = S(z_1,z_2,z_3) +  \left( h, k , 0 \right).
\end{align*}

Conversely, the above formulae give a fixed point free action of the dihedral group $G= D_4$ which contains no translations.

In particular, we have the following normal form:

$$ E = \CC / (\ZZ + \ZZ \tau), \ \  E' = \CC / (\ZZ + \ZZ \tau'), \ \ \tau, \tau' \in \HHH : = \{ z \in \CC| Im ( z) > 0 \}, $$
 $$h = 1/2, k =  \tau /2 , h' = 1/4$$
$$ r (z_1, z_2,z_3) : = (z_2, - z_1, z_3 + 1/4) $$
$$ s (z_1, z_2, z_3) : = (z_1 + 1/2 , - z_2 + \tau /2, - z_3 ) .$$

In particular, the Teichm\"uller space of hyperelliptic threefolds with group $D_4$  is isomorphic to the product $\HHH^2$
of two upper halfplanes.
\end{thm}

\section{Proof of the main theorem}

We use the following notation: $T = V/\Lam$ is  a complex torus of dimension $3$, which admits a free action of the group 
$$G = \langle r,s | r^4 = s^2 = (rs)^2 = 1\rangle \cong D_4,$$ such that the complex representation $\rho \colon G \to \GL(3, \CC)$ is faithful. 

A first observation is that the complex representation $\rho$ of $G$ must contain the $2$-dimensional
irreducible  representation $V_1$ of $G$ (else, $\rho$ would be a direct sum of 1-dimensional representations: this, by the assumption
on the faithfulness of $\rho$, would imply that $G$ is Abelian, a contradiction). 

Hence we have a splitting 
$$V = V_1 \oplus V_2,$$ where $V_2 $ is 1-dimensional, and we can choose an appropriate basis so that,
setting   $R := \rho(r), S := \rho(s)$,   we are left with the two cases

\begin{align*}
&\textbf{Case 1:} \;\; R = \begin{pmatrix}
0 & 1 & \\
-1 & 0 & \\
&& 1
\end{pmatrix},\;\;\; S = \begin{pmatrix}
1 & & \\
 & -1 & \\
&& -1
\end{pmatrix}, \\
&\textbf{Case 2:} \;\; R = \begin{pmatrix}
0 & 1 & \\
-1 & 0 & \\
&& 1
\end{pmatrix},\;\;\; S = \begin{pmatrix}
1 & & \\
& -1 & \\
&& 1
\end{pmatrix}.
\end{align*}
which are distinguished by the multiplicity of the eigenvalue $1$ of $S$. 

Indeed  $R$ is necessarily of the form above, since the freeness of the $G$-action implies that $\rho(g)$ must have eigenvalue
$1$ for every $g \in G$. \\

\begin{lemma}
In both Cases 1 and 2, the complex torus $T = V/\Lam$ is isogenous to a product of three elliptic curves, $T \sim_{\text{isog.}} E_1 \times E_2 \times E_3$, where $E_i \subset T$, for $i=1,2,3$ and $E_1$ and $E_2$ are isomorphic elliptic curves. In other words, writing $E_j = W_j/\Lam_j$, the complex torus $T$ is isomorphic to $$(E_1 \times E_1 \times E_3)/H, \ \ H =  \Lam/(\Lam_1 \oplus \Lam_2 \oplus \Lam_3).$$
\end{lemma}

\begin{proof}
Let $I$ be the identity of $T$.

In Case 1, we set $E_1 := \ker(S-I)^0 = \im(S+I)$, $E_3 := \ker(R-I)^0$ and $E_2 := R(E_1)$ (here, the superscript zero denotes the connected component of the identity). Then it is clear that $E_1 \cong E_2$, and that $T$ is isogenous to $E_1 \times E_2 \times E_3$. \\
In Case 2, we define similarly $E_2 := \ker(S+I)^0 = \im(S-I)$, $E_3 := \ker(R-I)_0$ and $E_1 := R(E_2)$. We obtain again $E_1 \cong E_2$, and that $T$ is isogenous to $E_1 \times E_2 \times E_3$.

\end{proof}

\begin{lemma} Writing $E_j = W_j/\Lam_j$, the following statements hold.
\begin{itemize}
	\item[(1)] In Case 1, the lattice $\Lam_2$ is equal to $W_2 \cap \Lam$.
  \item[(2)] In Case 2, the lattice $\Lam_1$ is equal to $W_1 \cap \Lam$.  
	\end{itemize}
\end{lemma}

\begin{proof}
(1) Obviously, $E_2 = R(E_1) = W_2/R(\Lam_1)$, i.e., $\Lam_2 = R(\Lam_1) \subset W_2 \cap \Lam$. On the other hand, $R(W_2 \cap \Lam) \subset W_1 \cap \Lam = \Lam_1$, and applying the automorphism $R$ of $\Lam$ gives $W_2 \cap \Lam \subset R(\Lam_1) = \Lam_2$. \\

(2) Here, $E_1 = R(E_2) = W_1/R(\Lam_2)$, i.e., $\Lam_1 = R(\Lam_2) \subset W_1 \cap \Lam$. For the converse inclusion, observe $R(W_1 \cap \Lam) \subset W_2 \cap \Lam = \Lam_2$, and applying $R$ yields again the result.

\end{proof}

We can now choose coordinates on $V$ such that $r$ is induced by a transformation of the form 
$$ r (z_1, z_2,z_3)  = (z_2, - z_1, z_3 + c_3) ,$$
by choosing as the origin in $V_1$ a fixed point of the restriction of $r$ to $V_1$.

We can now view $r,s$ as affine self maps of $T$ induced by affine self maps of $E_1 \times E_2 \times E_3$
of the form 
$$ r (z_1, z_2,z_3)  = (z_2, - z_1, z_3 + c_3) ,$$
$$ s (z_1, z_2, z_3) : = (z_1 + a_1 , - z_2 + a_2, \pm  z_3 + a_3) ,$$
and sending the subgroup $H$ to itself.

\begin{lemma}\label{r-free}
 The freeness of the action of the powers of $r$ is equivalent to:
 $H$ contains no element with last coordinate equal to $c_3$, or $2 c_3$.
 
 Moreover, $(0,0,4c_3) \in H$.
 \end{lemma}
 \begin{proof}
 $r(z) = z$ is equivalent to $(z_1 - z_2, z_1 + z_2 , - c_3) \in H$.
 However, the endomorphism  
 $$(z_1, z_2) \mapsto (z_1 - z_2, z_1 + z_2 )$$
 of $E_1 \times E_2$ is surjective, hence $H$ cannot contain any element with last coordinate equal to $c_3$.
 
 Since $r^2(z) = (-z_1, -z_2, z_3 + 2 c_3)$,  
 $r^2(z) = z$ is equivalent to $(-2z_1 , -2 z_2 , 2 c_3) \in H$, and we reach the similar conclusion that
 $H$ cannot contain any element with last coordinate equal to $2c_3$.
 
 Finally, the condition that $r^4$ is the identity is equivalent to $(0,0,4c_3) \in H$.

 \end{proof}

\begin{prop}
 Case 2 does not occur.
 \end{prop}
 \begin{proof}

Since we assume that 
$$ s (z_1, z_2, z_3) : = (z_1 + a_1 , - z_2 + a_2,   z_3 + a_3) ,$$
and that $s^2$ is the identity, it must be $$(2a_1, 0 , 2 a_3) \in H.$$

Consider now $rs$:
$$ rs(z) = ( -z_2 + a_2, - z_1 - a_1, z_3 + a_3 + c_3). $$

The condition that $(rs)^2$ is the identity is equivalent to:
$$(a_1 + a_2, -(a_1 + a_2), 2 ( a_3 + c_3)) \in H.$$

This condition, plus the previous one, imply that 
$$(a_2 - a_1, - (a_1 + a_2), 2  c_3) \in H,$$
contradicting Lemma \ref{r-free}.

 \end{proof}
 
 Henceforth we shall assume that we are in Case 1, and we can choose the origin in $E_3$ so that 
 $$ s (z_1, z_2, z_3) : = (z_1 + a_1 , - z_2 + a_2,   - z_3 ) .$$
 
 \begin{lemma}\label{s-free}
 If  $$ s (z_1, z_2, z_3) : = (z_1 + a_1 , - z_2 + a_2,   - z_3 ),$$
 then $$(2a_1, 0 ,0) \in H$$ and $H$ contains no element of the 
 form $$(a_1, w_2,w_3).$$

 \end{lemma}
 \begin{proof}
 The first condition is equivalent to $s^2$ being the identity, while the second
 is equivalent to the condition that $s$ acts freely, since
 $s(z) = z$ is equivalent to $(a_1 , -2 z_2 + a_2 , - 2 z_3) \in H$.

 \end{proof}
 
\begin{prop}
For each $\la \in \Lam$ there exist $\la' \in \Lam, \la_1 \in \Lam_1, \la_2 \in \Lam_2, \la_3 \in \Lam_3,$,
such that
$$  2 \la = \la_1 + \la', \ \ 2 \la' = \la_2 + \la_3 $$

More precisely, we even have: $$\Lam \subset (1/2 ) \Lam_1 + (1/2 ) \Lam_2 + (1/4)  \Lam_3.$$

\
\end{prop}

\begin{proof}
Let $\la \in \Lam$: we can write $$2\la = \underbrace{(I + S)\la}_{=: \la_1 \in \Lam_1} + \underbrace{(I - S)\la}_{=: \la' \in \Lam}.$$
Furthermore, since $\la' \in \im(I - S)$, we obtain $$2\la' = \underbrace{(I + R^2)\la'}_{=: \la_3 \in \Lam_3} + \underbrace{(I - R^2)\la'}_{=: \la_2 \in \Lam \cap W_2 = \Lam_2}.$$
Hence, $\la = \frac{\la_1}2 + \frac{\la_2}4 + \frac{\la_3}4$ for unique $\la_j \in \Lam_j$. 

Applying the automorphism $R$ of $\Lam$ and the unicity of the $\la_j$ yields the result, since $R$ exchanges $\Lam_1$ and $\Lam_2$.\\

\end{proof}

\begin{prop}
We have $$\Lam \subset (1/2 ) \Lam_1 + (1/2 ) \Lam_2 + (1/2)  \Lam_3.$$

\end{prop}
\begin{proof}
For $\la \in \Lam$ we can write  $\la = \frac{\la_1}2 + \frac{\la_2}2 + \frac{\la_3}4$ for unique $\la_j \in \Lam_j$. 

We now use the property
$$  E_i \hookrightarrow T \Rightarrow \forall \ (0,0,d) \in H , \  d = 0.$$
Indeed, $2 \la = \la_1 + \la_2 + \frac{\la_3}2$, hence  $(0,0,[\frac{\la_3}2]) \in H$ and $\frac{\la_3}2 = 0$
in $E_3$. Equivalently, there is an element $\la'_3 \in \Lam_3$ with $$\frac{\la_3}4 = \frac{\la'_3}2.$$

\end{proof}

\begin{lemma}\label{rs}
Consider  the transformation $rs$:
$$  rs (z) = (-z_2 + a_2 , -z_1 - a_1, -z_3  + c_3).$$ 

The condition that its square is the identity amounts to

$$ (a_1 + a_2, - ( a_1 + a_2 ) , 0) \in H,$$
while the freeness of its action is equivalent to the fact that $H$ contains no element 
of the form 
$$ (w_1 - a_2, w_1 + a_1, w_3) \Leftrightarrow \forall \ (d_1,d_2, d_3) \in H \colon \;\; d_1 + a_2 \neq d_2 -a_1.$$
\end{lemma}

\begin{proof}
The first condition is straighforward, while the freeness of the action is equivalent to the non existence of $(z_1, z_2, z_3)$
such that
$$  (z_1 + z_2 - a_2, z_2 + z_1 + a_1, 2 z_3 - c_3) \in H.$$
As usual, we observe that for each $w_1, w_3$ there exist $z_1 , z_2, z_3 $ with $z_1 + z_2 = w_1, 2 z_3 - c_3 = w_3.$

\end{proof}

We put together  the conclusions of Lemmas \ref{r-free}, \ref{s-free}, \ref{rs}, 

\begin{itemize}
\item
(i) $(0,0,4c_3) \in H$
\item
(ii) $(2 a_1,0,0) \in H$
\item
(iii) $( a_1+ a_2,-a_1 - a_2,0) \in H$, hence also $( a_1- a_2,a_1 +a_2,0) \in H$.

\end{itemize}

\begin{enumerate}
\item
$H$ contains no element of the form
$(w_1,w_2,c_3) $, 
\item
nor of the form $(w_1,w_2,2 c_3) $
\item
nor of the form $(a_1,w_2,w _3) $
\item
nor of the form $(w_1,w_2,w_3) $ with $w_1 + a_2 = w_2 - a_1$.

\end{enumerate}

It follows from (iii) and (3) that $a_2 \neq 0$. While the condition that each element of $H$ which has
two coordinates equal to zero is indeed zero (since $E_i $ embeds in $T$!) imply 
$$  2 a_1 = 0, 4 c_3 = 0.$$

By conditions (1), (2), (3) the elements $a_1$, $c_3$ have respective orders exactly $2,4$. Moreover:

\begin{itemize}
\item
 (4) and (i) imply that
 $a_1 + a_2 \neq 0$
\item
(ii), (iii) and the fact that $H$ has exponent 2 implies $2a_2= 2a_1=0 $, $2a_1 + 2a_2 = 0$.
Hence $a_1 \neq  a_2$ are nontrivial 2-torsion elements.

\end{itemize}

We have thus obtained the desired elements
$$ h : = a_1, k : = a_2, h' : = c_3.$$

It suffices to show that $H$ is generated by $\om := (h + k,h + k,0) = (a_1 + a_2, a_1+ a_2, 0)$.

Observe first that $\om \in H$, by condition (iii). 

Condition (4) implies that the first coordinate of an element of $H$ must be a multiple of $(a_1 + a_2)$: since it cannot equal
$a_1$, by condition (3), and if it equals $a_2$, we can add $\om$ and obtain an element of $H$ with first coordinate $a_1$.
Using $R$, we infer that both coordinates must be a multiple of $(a_1 + a_2)$. Possibly adding $\om$, we may assume that $w_1 = 0$:
then by (4) we conclude that also $w_2 = 0$. Finally, the condition that each element of $H$ which has
two coordinates equal to zero is indeed zero, show that $H$ is then generated by $\om$, as we wanted to show.

The last assertions  of the main theorem follow now in a straightforward way (see \cite{Catanese-Corvaja} concerning general properties of
Teichm\"uller spaces of hyperelliptic manifolds).

\end{document}